\documentclass[smallextended,envcountsect]{svjour3}
\usepackage{amsmath}
\usepackage{amsfonts}
\usepackage{mathptmx}
\usepackage{amssymb, graphicx}%
\usepackage{color,cite}
\usepackage{hyperref}
\hypersetup{
    colorlinks=true,
    linkcolor=blue, 
    citecolor=red, 
    urlcolor=blue  } 
\usepackage[colorinlistoftodos]{todonotes}
\def\text#1{\mbox{#1}}

\newcommand{\la}{\lambda}
\newcommand{\de}{\delta}
\newcommand{\eps}{\varepsilon}
\newcommand{\bx}{\bar x}

\newcommand {\R} {\mathbb R}

\newcommand {\B} {\mathbb B}

\newcommand {\dom} {{\rm dom}\,}

\newcommand {\bd} {{\rm bd}\,}

\newcommand {\Er} {{\rm Er}\,}

\newcommand {\sd} {\partial}
\newcommand {\Int} {{\rm int}\,}

\newcommand{\gfrac}[2]{\genfrac{}{}{0pt}{}{#1}{#2}}

\def\RHS{right-hand side}
\def\SVM{set-valued mapping}

\newcommand{\ang}[1]{\left\langle #1 \right\rangle}

\newcounter{mycount}

\newcommand\xqed{%
  \leavevmode\unskip\penalty9999 \hbox{}\nobreak\hfill
  \quad\hbox{$\triangle$}}


\smartqed

\def\Ptb{{\rm Ptb\,}}

\title{Perturbation of error bounds
\thanks{The research is supported by the Australian Research Council: project DP160100854;
EDF and the Jacques Hadamard Mathematical Foundation: Gaspard Monge Program for Optimization and Operations Research.
The research of the second and third authors is also supported by MINECO of Spain
and FEDER of EU: grant MTM2014-59179-C2-1-P. }}
\dedication{Dedicated to Professor Terry Rockafellar in honor of his 80th birthday}

\author{A. Y. Kruger
\and M. A. L\'{o}pez
\and M. A. Th\'{e}ra}

\institute{A. Y.~Kruger \at
Centre for Informatics and Applied Optimization, Faculty of Science and Technology, Federation University Australia\\
\email{a.kruger@federation.edu.au}
\and
M. A. L\'{o}pez \at
Department of Statistics and Operations Research, University of Alicante, Spain and Federation University Australia\\
\email{marco.antonio@ua.es}
\and
M. A. Th\'{e}ra \at
Laboratoire XLIM, UMR-CNRS 6172, University of Limoges, France and Federation University Australia\\
\email{michel.thera@unilim.fr}
\date{Received: date / Accepted: date}
}
\begin{document}

\maketitle

\begin{abstract}
Our aim in the current article is to extend the developments in Kruger, Ngai \& Th\'era, SIAM J. Optim. {\bf 20}(6), 3280--3296 (2010) and, more precisely,  to characterize, in the Banach space setting, the stability of the local and global error bound property of inequalities determined by lower semicontinuous functions under data perturbations.
We propose new concepts of (arbitrary, convex and linear) perturbations of the given function defining the system under consideration, which turn out to be a useful tool in our analysis.
The characterizations of error bounds for families of perturbations can be interpreted as estimates of the `radius of error bounds'.
The definitions and characterizations are illustrated by examples.

\end{abstract}

\section{Introduction and Preliminaries}
In this article, we mostly follow the standard terminology and notation used by the optimization and variational analysis community, see,  e.g., \cite{RocWet98,DonRoc14}.
Throughout, $X$ and $X^*$ stand for a real Banach space and its topological dual, respectively, while $\B$ and $\B^*$ denote the corresponding unit balls.
The norms in both spaces are denoted by the same symbol $\Vert\cdot\Vert$.
For a subset $S$, we denote its interior and boundary by $\Int S$ and $\bd S$, respectively.
The distance from a point $x$ to a set $S$ is denoted by $d(x,S):=\inf_{u\in S}\Vert u-x\Vert$, and we use the con\-vention $d(x,S)=+\infty$  whenever   $S=\emptyset$.

For an extended-real-valued  function: $f : X\rightarrow\R_\infty:=\R\cup\{+\infty\}$, its domain is the set
$\dom f:=\{x\in X\mid f(x) < +\infty\}$.
A function $f$ is said to be \emph{proper} if
$\dom f\ne\emptyset$.
We use the symbol $f_+(x)$ to denote $\max(f(x),0)$.
The class of all extended-real-va\-lued proper convex lower semicontinuous functions on $X$ is denoted by $\Gamma_0(X)$.
For a convex function $f:X\to {\R_\infty}$, its (Moreau) \emph{subdifferential} at $x\in\dom f$ is given by
$$\partial  f(x):=\{x^*\in X^*\mid  \langle x^{\star},u-x\rangle\le f(u)-f(x),\; \forall u\in X\}.$$
It is a (possibly empty) weak$^*$-closed set in $X^*$.

The article is  concerned with the study of the solution set of a single inequality of the type
\begin{equation}\label{2}
S_f:=\{x\in X\mid f(x)\le 0\},
\end{equation}
where $f:X\rightarrow\R_\infty$.
Such sets subsume, e.g., feasible sets in mathematical programming.
Indeed, solutions of a finite family of inequalities:
\begin{equation*}\label{l2bis}
\{x\in X\mid f_i(x)\leq 0\quad \text{ for all} \quad i=1,\ldots, n\}
\end{equation*}
can be re\-written as (\ref{2}) with function $f$ defined by
$f(x)=\max\{f_1(x),\ldots,f_n(x)\}$.

An important issue when studying  systems  of the type (\ref{2}) is to give an upper estimate (error bound) of the distance from a point $x\in X$  to the set $S_f$ in terms of a  computable function measuring the violation of the inequality in \eqref{2}.
This can, e.g., be the function $f$ itself.

We say that a function $f$ has (or admits) a {\it local error bound} at a point $\bar x \in S_f$ if there exists a real $\tau>0$ such that
\begin{equation}\label{4}
\tau d(x,S_f) \leq f_+(x)
\end{equation}
for all $x$ near $\bar x$.
We similarly say that a function $f$ has a \textit{global error bound} if there exists a real $\tau>0$ such that inequality (\ref{4}) is satisfied for all $x\in X$.

The exact upper
bound of all such $\tau$ (the \textit{error bound modulus}; cf. \cite{FabHenKruOut10}) equals either \begin{equation}\label{ErL}
\Er{f}(\bar{x}) :=\liminf_{\gfrac{x\to\bar{x}}{f(x)>0}}
\frac{f(x)}{d(x,S_f)}
\end{equation}
in the local setting, or
\begin{equation}\label{ErG}
\Er{f}:=\inf_{f(x)>0} \frac{f(x)}{d(x,S_f)}
\end{equation}
in the global case.
Constants (\ref{ErL}) and (\ref{ErG}) provide
quantitative estimates of the error bound property.

The starting point of the theory  of error bounds goes back to the pio\-neering work  by Hoffman \cite{Hof52} (although some traces of the error bound property can be found in an earlier publication by Rosenbloom \cite{Ros51}\footnote{Private communication by J.-B. Hiriart-Urruty})
who established for a linear function in finite di\-mensions the following result:
\vskip 2mm
\textit{ Given an $m\times n$  matrix $A$ and a vector $b\in \R^m$, there exists a positive number $\kappa > 0 $ such that the distance from $x$ to the set $S:=\{x\in\R^n\mid  Ax\leq b\} $  has an upper bound given  by $\kappa \Vert (Ax-b)_+\Vert$, where  for  $y:=(y_1,\ldots,y_m)\in \R^m$,  $y_+$ denotes the vector  $(\max(y_1,0),\ldots,\max(y_m,0))$.}
\vskip 2mm


After the work by Hoffmann and its extensions by Robinson \cite{Rob75}, Man\-gasarian \cite{Man85}, Auslender \& Crouzeix \cite{AusCro88}, Pang \cite{Pang97}, Lewis and Pang \cite{LewPan98}, Klatte \& Li \cite{KlaLi99}, Jourani \cite{Jou00}, there have been significant developments of various aspects of errors bounds for con\-vex and nonconvex functions in recent years.
The interested reader is referred to the articles by Ng \& Zheng \cite{NgZhe01}, Az{\'e} \cite{Aze03,Aze06}, Az{\'e} \& Corvellec \cite{AzeCor04}, Z{\u{a}}linescu \cite{Zal03}, Huang \& Ng \cite{HuaNg04}, Corvellec \& Motreanu \cite{CorMot08}, Fabian et al \cite{FabHenKruOut10,FabHenKruOut12}, Gfrerer \cite{Gfr11}, Ioffe \cite{Iof00_,Iof16}, Ioffe \& Outrata \cite{IofOut08}, Ngai \& Th{\'e}ra \cite{NgaThe04,NgaThe05,NgaThe08,NgaThe09}, Zheng \& Ng \cite{ZheNg10,ZheNg12}, Bednarczuk \& Kruger \cite{BedKru12,BedKru12.2}, Meng \& Yang \cite{MenYan12}, Kruger \cite{Kru15,Kru15.2,Kru15.3} and the references therein.

Many authors have recently studied
error bounds   in connection with the \emph{metric regularity} and \emph{subregularity} (cf. \cite{DonRoc14}) as well as \emph{Aubin property} and \emph{calmness} of set-valued mappings: \cite{Aze06,CanKruLopParThe14,CanHenLopPar16,Gfr11,Iof00_,Iof16, IofOut08, NgaThe04,NgaThe08,Kru15,Kru15.2,Kru15.3,ZheNg10,ZheNg12, NgaTroThe13,NgaTroThe14}.
The connections between the error bounds and weak sharp minima were studied in \cite{BurDeng05}.

Another typical example where the theory of error bounds plays an important role  is the so-called \textit{feasibility problem} \cite{BauBor96, BecTeb03, HesLuk13, KruLukTha}, which consists in finding a point in the intersection of a finite family  of  closed (usually convex) sets and has a broad applicability in various areas such as, e.g., image reconstruction \cite{Com96}.
Several iterative methods such as the method of successive orthogonal projections, the cyclic subgradient projections method, etc,  are known to solve this problem (see \cite{Cen84}).
Error bounds are also used in the convergence analysis of projection algorithms.
They allow one to quantify the proximity of an iterate to the solution set of the problem.
The reader is referred to the recent survey paper by Ioffe \cite{Iof16,Iof16.2} and the references therein and also to some recent contributions  by Beck \& Teboulle \cite{BecTeb03}, Lewis et al \cite{LewLukMal09}, Hesse \& Luke \cite{HesLuk13}, Borwein et al \cite{BorLiYao14}, Drusvyatskiy et al \cite{DruIofLew15.2}, Kruger \& Thao \cite{KruTha16}, Noll \& Rondepierre \cite{NolRon16}, and Bolte et al \cite{BolNguPeySut}.

In this article, we
study stability of local and global error bounds under perturbations.
It was  observed in Ngai et al \cite[Theorem~1]{NgaKruThe10} and Kruger et al \cite[Theorem~1]{KruNgaThe10}, that the requirement for the distance from the origin to the subdifferential of a function at a reference point to be strictly positive, while being a conventional sufficient condition for a local error bound for the system (\ref{2}), is far from being necessary.
This con\-dition guaranties the local error bound property not just for the given function, but also for a family of its small per\-turbations.
{In the setting of \SVM s and the metric subregularity property, similar observations have been made recently by Gfrerer \cite[Theorem~3.2]{Gfr11} and Li \& Mordukhovch \cite[Theorem~4.4]{LiMor12} in terms of the distance from the origin to the \emph{critical limit set} \cite[Definition~3.1]{Gfr11} and the kernel of the \emph{reversed mixed coderivative} \cite{Mor06.1}, respectively.
We also mention the most recent development in Gfrerer \& Outrata \cite[Theorem~2.6]{GfrOut16} where \cite[Theorem~3.2]{Gfr11} has been upgraded (in finite dimensions) to an almost radius-type result.}

In Section~\ref{On}, we exploit the concept of local $\varepsilon$-perturbation from \cite{KruNgaThe10}, define subfamilies of convex and linear $\varepsilon$-perturbations, and establish in Theorem~\ref{ge} the exact formula for the `radius of  local error bounds' for families of  arbitrary, convex and linear perturbations in terms of the distance from the origin to the subdifferential.
Theorem~\ref{ge} in a sense continues the series of \emph{radius theorems} from Dontchev--Ro\-ckafellar \cite{DonRoc14}, ``furnishing
a bound on how far perturbations of some sort in the specification of a problem can
go before some key property is lost'' \cite[page~364]{DonRoc14}.
Note that, unlike the `stable' properties of \emph{metric regularity}, \emph{strong metric regularity} and \emph{strong metric subregularity} studied in \cite{DonRoc14}, the error bound property (as well as the equivalent to it metric subregularity property, cf. \cite[page~200]{DonRoc14}) can be lost under infinitely small perturbations.
That is why stronger conditions are required to ensure stability.
Unlike, e.g., \cite[Theorem~6A.9]{DonRoc14} which gives the radius of strong metric subregularity in terms of the \emph{subregularity modulus} being an analogue of (the reciprocal of) the error bound modulus $\Er{f}(\bar{x})$ \eqref{ErL}, Theorem~\ref{ge} utilizes the quantity ${|\partial{f}|{}_{\rm bd}(\bar{x})}$ (see Theorem~\ref{lo} below).
In view of Theorem~\ref{lo}, this quantity can be smaller than $\Er{f}(\bar{x})$, and, thus, condition ${|\partial{f}|{}_{\rm bd}(\bar{x})}>0$ used in Corollary~\ref{re} imposes a strong requirement on the function $f$ and ${|\partial{f}|{}_{\rm bd}(\bar{x})}$ leads to a radius theorem of a different type compared to those in \cite{DonRoc14}.

The same idea applies partially when considering stability of global error bounds in Section~\ref{S3}.
In the global setting, we define families of convex and linear perturbations as well as larger families of weak convex and weak linear perturbations, discuss some limitations of these definitions, and establish in Theorem~\ref{Ge} lower and upper estimates for the `radius of global error bounds'  for families of such perturbations.
The family of convex perturbations considered here is larger than the corresponding one studied in \cite{KruNgaThe10}.
In particular, the so called \emph{asymptotic
qualification condition} is waived.

Some examples are given for the convenience of the reader to illustrate the different concepts introduced along the presentation.

\section{Stability of local error bounds}\label{On}

In this section we establish conditions for stability of local error bounds for the con\-straint  system (\ref{2}).
We start with the  following statement extracted from \cite[The\-orem~1]{KruNgaThe10}.


\begin{theorem}\label{lo}
Let $f\in\Gamma_0(X)$ and $f(\bar{x})=0$.
Then function $f$ has a local error bound at $\bar x$, provided that one of the following two conditions is satisfied:
\begin{enumerate}
\item
${\overline{|\partial{f}|}{}^>(\bar{x})}:=\liminf_{x\to \bar{x},\,
f(x)>f(\bar{x})}d(0,\partial{f}(x))>0$;
\item
${|\partial{f}|{}_{\rm bd}(\bar{x})}:=d(0,\bd\partial{f}(\bar{x}))>0$.
\end{enumerate}
Moreover, condition {\rm (i)} is also necessary for $f$ to have a local error bound at $\bar x$ and
\begin{equation}\label{lo.a}
{|\partial{f}|{}_{\rm bd}(\bar{x})}\le{\overline{|\partial{f}|}{}^>(\bar{x})}=\Er{f}(\bar{x}).
\end{equation}
\end{theorem}

Constant ${\overline{|\partial{f}|}{}^>(\bar{x})}$ is known as the \emph{strict outer subdifferential slope}
\cite{FabHenKruOut10}
of $f$ at $\bar{x}$.
{We are going to call $|\partial{f}|{}_{\rm bd}(\bar{x})$ the \emph{boundary subdifferential slope}
of $f$ at $\bar{x}$.}
The sufficient criterion (i) was used in
\cite[Theorem~2.1~(c)]{IofOut08}, \cite[Corollary~2~(ii)]{NgaThe09},
\cite[Theorem~4.12]{Pen10}, \cite[Theorem 3.1]{WuYe01}. Criterion (ii) was used in \cite[Corollary~3.4]{HenJou02},
\cite[Theorem~4.2]{HenOut01}.
The equality in (\ref{lo.a}) is well known.
See also
characterizations of linear and nonlinear conditionings in
\cite[Theorem~5.2]{CorJouZal97}.

The inequality in (\ref{lo.a}) can be strict.

\begin{example}\label{gl}
1. $f(x)\equiv0$, $x\in\R$. Obviously $0\in\bd\partial{f}(\bar{x})$,
${|\partial{f}|{}_{\rm bd}(\bar{x})}=0$, while ${\overline{|\partial{f}|}{}^>(0)}=\infty$ for any $\bar{x}\in\R$.
\sloppy

2. $f(x)=0$ if $x\le0$, and $f(x)=x$ if $x>0$. Then
$\partial{f}(0)=[0,1]$ and $0\in\bd\partial{f}(0)$,
${|\partial{f}|{}_{\rm bd}(0)}=0$, while ${\overline{|\partial{f}|}{}^>(0)}=1$.

3. If $f:\mathbb{R}^{m}\rightarrow \mathbb{R}$ is
convex, thanks to \cite[Theorem 3.1]{CanHanParTol15}, \begin{equation*}
\mathrm{bd\ }\partial f\left( \overline{x}\right) =\ \underset{x\rightarrow
\overline{x},\;x\neq \overline{x}}{\limsup }\partial f\left( x\right) ,
\end{equation*}%
and ${|\partial{f}|{}_{\rm bd}(\bar{x})}$ can be easily computed:
\begin{equation*}
{|\partial{f}|{}_{\rm bd}(\bar{x})}=\underset{x\rightarrow \overline{x},\;x\neq
\overline{x}}{\liminf }d\left( 0,\partial f\left( x\right) \right) .
\end{equation*}

In the particular case when $f$ is
a polyhedral function
$$f(x):=\max_{i=1,\ldots,n}\left(\langle a_{i},x\rangle -b_{i}\right)$$
with
$a_{i}\in \mathbb{R}^{m},$ $b_{i}\in\R$ $(i=1,\ldots,n),$ thanks to \cite[Theorem 3.1]{CanHenLopPar16},
\begin{equation*}
\underset{x\rightarrow \overline{x},\;f\left( x\right) >f\left( \overline{x}%
\right) }{\limsup }\partial f\left( x\right) =\bigcup\limits_{D\in
\mathcal{D}\left( \overline{x}\right) }\mathrm{conv}\left\{ a_{i},\text{ }%
i\in D\right\} ,
\end{equation*}%
where $\mathcal{D}\left( \overline{x}\right) $ denotes the family of all
subsets
$$D\subset I\left( \overline{x}\right) :=\{i=1,\ldots,n\mid f_{i}\left(
x\right) =f\left( x\right) \},$$
such that the system
\begin{equation*}
\left\{
\begin{tabular}{rl}
$\left\langle a_{i},d\right\rangle =1,$ & $i\in D,$ \\
$\left\langle a_{i},d\right\rangle <1,$ & $i\in I\left( \overline{x}\right)
\setminus D$%
\end{tabular}%
\right\}
\end{equation*}%
is consistent (in the variable $d\in \mathbb{R}^{m})$. In other words, $D\in
\mathcal{D}\left( \overline{x}\right) $ if there exists a hyperplane
containing $\{a_{i},$ $i\in D\}$ and such that
\begin{equation*}
\{0\}\cup \{a_{i},i\in I\left( \overline{x}\right) \setminus D\}
\end{equation*}%
lies in one of the open half-spaces determined by this hyperplane.
Hence,
$${\overline{|\partial{f}|}{}^>(\bar{x})}=d\left(0,\bigcup\limits_{D\in
\mathcal{D}\left( \overline{x}\right) }\mathrm{conv}\left\{ a_{i},\text{ }%
i\in D\right)\right).$$

For a function $f:\mathbb{R}^{m}\rightarrow \mathbb{R}$ which is
regular and locally Lipschitz continuous at $\overline{x}$, \cite[Theorem 3.1]{LiMenYan}
 provides lower and upper bounds for \textrm{Er\ }$f(%
\overline{x})$ involving outer limits of subdifferentials of the function $f$
at $\overline{x}$ and the support function of $\partial f\left( \overline{x}%
\right) $ at $0$, respectively.
\xqed\end{example}

Thus, condition (ii) in Theorem~\ref{lo} is in general stronger than condition~(i).
{It imposes restrictions on the behaviour of the function $f$ near $\bx$ not only outside the set $S_f$, but also inside it, particularly excluding the situations as in parts 1 nd 2 of Example~\ref{gl}.
Condition (ii) characterizes a stronger property than just the existence of a local error bound for $f$ at $\bar{x}$ which can be of interest by itself.
In particular,}
it guaranties the
local error bound property for the family of functions being small
perturbations of $f$ in the sense defined below.

Let $f(\bar{x})<\infty$ and $\varepsilon\ge0$.
Following \cite[Definition 5]{KruNgaThe10}, we say that
$g:X\to\R_\infty$ is an \emph{$\varepsilon$-perturbation} of $f$ near $\bar{x}$ if
\begin{gather}\label{io-}
g(\bar{x})=f(\bar{x})
\end{gather}
and
\begin{gather}\label{io}
\limsup_{x\to\bar{x}} \frac{|g(x)-f(x)|}{\|x-\bar{x}\|}\le\varepsilon.
\end{gather}
If both functions $f$ and $g$ are continuous at $\bar{x}$, then equality \eqref{io-} in the above definition is obviously implied by condition \eqref{io}.

The collection of all $\varepsilon$-perturbations of $f$ near $\bar{x}$ will be denoted by $\Ptb(f,\bar{x},\varepsilon)$.
Obviously, if $g\in\Ptb(f,\bar{x},\varepsilon)$, then
$f\in\Ptb(g,\bar{x},\varepsilon)$, and neither $f$ nor $g$ are required to be convex.
If both $f$ and $g$ are convex, then the actual perturbation function $p:=g-f$ needs not to be convex; it is in general a d.c. function (difference of convex functions).

The following subsets of $\Ptb(f,\bar{x},\varepsilon)$ corresponding, respectively, to convex and linear $\varepsilon$-perturbations of $f$ near $\bx$ can be of interest:
\begin{align}\label{rba}
\Ptb_c(f,\bar{x},\varepsilon):=& \{g\in\Ptb(f,\bar{x},\varepsilon)\mid g-f\in\Gamma_0(X)\},
\\\notag
\Ptb_l(f,\bar{x},\varepsilon):=& \{g\mid g(u)-f(u)=\ang{x^*,u-\bx} (u\in X),\; x^*\in\eps\B^*\}.
\end{align}
Obviously,
\begin{align}\label{rep}
\Ptb_l(f,\bar{x},\varepsilon) \subset \Ptb_c(f,\bar{x},\varepsilon) \subset
\Ptb(f,\bar{x},\varepsilon).
\end{align}

The next proposition provides a sufficient condition for a function $g$ to be a convex $\varepsilon$-perturbation of $f$ near $\bar{x}$.

\begin{proposition}
Let $f(\bar{x})<\infty$ and $\varepsilon\ge0$.
Suppose $g=f+p$ where $p:X\to\R_\infty$ is convex
and Lipschitz continuous near $\bx$ with constant $\eps$ and
$p(\bx)=0$.
Then $g\in\Ptb_c(f,\bar{x},\varepsilon)$.
\end{proposition}
\begin{proof}
Thanks to the Lipschitz continuity of $p$,
$\sd p(\bx)\ne\emptyset$
and $\|x^*\|\le\eps$ for all $x^*\in\sd p(\bx)$;
{cf. \cite[Proposition~4.1.25 and its proof]{BorVan10}.}
Given an $x^*\in\sd p(\bx)$, for any $x\in X\setminus\{\bx\}$, we have:
\begin{equation*}
\frac{g(x)-f(x)}{\|x-\bar{x}\|}=\frac{p(x)}{\|x-\bar{x}\|} \ge\frac{\ang{x^*,x-\bx}}{\|x-\bar{x}\|}\ge-\eps.
\end{equation*}
On the other hand,
\begin{equation*}
\limsup_{x\to\bar{x}}\frac{g(x)-f(x)}{\|x-\bar{x}\|} =\limsup_{x\to\bar{x}}\frac{p(x)}{\|x-\bar{x}\|}
\le\eps.
\end{equation*}
Combining the two inequalities proves \eqref{io}.
\qed\end{proof}

The assumption
{of the Lipschitz continuity of $p$}
in the above proposition is essential.
\begin{example}
Let $f,p,g:\R\to\R_\infty$, $f(x)=0$ for all $x\in\R$ and
$$
p(x)=g(x)=
\begin{cases}
0& \mbox{if } x\le0,
\\
+\infty& \mbox{if } x>0.
\end{cases}
$$
Then $f,g,p\in\Gamma_0(X)$ and
\begin{equation*}
\limsup_{x\to0} \frac{|g(x)-f(x)|}{|x|}=+\infty.
\end{equation*}
\xqed\end{example}

Given a function $f\in\Gamma_0(X)$ with $f(\bar{x})=0$ and a number $\eps\ge0$, denote
\begin{align}\label{ErPtb}
\Er\{\Ptb(f,\bar{x},\varepsilon)\}(\bar{x}):=&
\inf_{g\in\Gamma_0(X)\cap\Ptb(f,\bar{x},\varepsilon)}\Er{g}(\bar{x}),
\\\notag
\Er\{\Ptb_c(f,\bar{x},\varepsilon)\}(\bar{x}):=&
\inf_{g\in\Ptb_c(f,\bar{x},\varepsilon)}\Er{g}(\bar{x}),
\\\notag
\Er\{\Ptb_l(f,\bar{x},\varepsilon)\}(\bar{x}):=&
\inf_{g\in\Ptb_l(f,\bar{x},\varepsilon)}\Er{g}(\bar{x}).
\end{align}
These numbers characterize the error bound property for families of $\eps$-per\-tur\-bations of $f$ near $\bar{x}$.
Thanks to \eqref{rep}, it holds
\begin{align}\notag
\Er\{\Ptb(f,\bar{x},\varepsilon)\}(\bar{x}) &\le\Er\{\Ptb_c(f,\bar{x},\varepsilon)\}(\bar{x})
\\\label{ert}
&\le\Er\{\Ptb_l(f,\bar{x},\varepsilon)\}(\bar{x})
\le\Er{f}(\bar{x})
\end{align}
for any $\eps\ge0$.

\begin{theorem}\label{ge}
Let $f\in\Gamma_0(X)$ and $f(\bar{x})=0$. Then
\begin{align}\notag
{|\partial{f}|{}_{\rm bd}(\bar{x})} &=\inf\{\eps>0\mid \Er\{\Ptb(f,\bar{x},\varepsilon)\}(\bar{x})=0\}
\\\notag
&=\inf\{\eps>0\mid \Er\{\Ptb_c(f,\bar{x},\varepsilon)\}(\bar{x})=0\}
\\\label{ge.01}
&=\inf\{\eps>0\mid \Er\{\Ptb_l(f,\bar{x},\varepsilon)\}(\bar{x})=0\}.
\end{align}
\end{theorem}

\begin{proof}
Thanks to the first two inequalities in \eqref{ert}, we always have
\begin{align}\notag
&\inf\{\eps>0\mid \Er\{\Ptb(f,\bar{x},\varepsilon)\}(\bar{x})=0\}
\\\notag
\le&\inf\{\eps>0\mid \Er\{\Ptb_c(f,\bar{x},\varepsilon)\}(\bar{x})=0\}
\\\label{ge.07}
\le&\inf\{\eps>0\mid \Er\{\Ptb_l(f,\bar{x},\varepsilon)\}(\bar{x})=0\}.
\end{align}
We are going to show that
\begin{align}\label{ge.02}
{|\partial{f}|{}_{\rm bd}(\bar{x})}\le\inf\{\eps>0\mid \Er\{\Ptb(f,\bar{x},\varepsilon)\}(\bar{x})=0\}
\end{align}
and
\begin{align}\label{ge.03}
\inf\{\eps>0\mid \Er\{\Ptb_l(f,\bar{x},\varepsilon)\}(\bar{x})=0\} \le{|\partial{f}|{}_{\rm bd}(\bar{x})}.
\end{align}
By \cite[Theorem~8~(i)]{KruNgaThe10},
$
\Er{g}(\bar{x}) \ge{|\partial{f}|{}_{\rm bd}(\bar{x})}-\varepsilon
$
for any $\varepsilon>0$ and any $g\in\Gamma_0(X)\cap\Ptb(f,\bar{x},\varepsilon)$.
By definition \eqref{ErPtb}, we have
\begin{gather*}
{|\partial{f}|{}_{\rm bd}(\bar{x})} \le\Er\{\Ptb(f,\bar{x},\varepsilon)\}(\bar{x})+\eps
\end{gather*}
for any $\varepsilon>0$, and consequently ${|\partial{f}|{}_{\rm bd}(\bar{x})} \le\eps$ if $\Er\{\Ptb(f,\bar{x},\varepsilon)\}(\bar{x})=0$, which yields inequality \eqref{ge.02}.

If $\varsigma(f,\bx)=\infty$, inequality \eqref{ge.03} holds trivially.
Let ${|\partial{f}|{}_{\rm bd}(\bar{x})}<\eps<\infty$.
We are going to show that
$\Er\{\Ptb_l(f,\bar{x},\varepsilon)\}(\bar{x})=0$.
By the definition of ${|\partial{f}|{}_{\rm bd}(\bar{x})}$, there exists an $x^*\in\bd\partial{f}(\bar{x})$ such that $\|x^*\|<\eps$.
Choose a positive number $\xi<\varepsilon -\Vert x^{\ast}\Vert$.
Thus,
\begin{gather}\label{ge.05}
f(u)\ge\langle x^*,u-\bx\rangle\quad \mbox{for all } {u}\in X,
\end{gather}
and there exists an $\hat x^*\in X^*$ such that $\|\hat x^*-x^*\|<\xi$ and $\hat x^*\notin\sd f(\bx)$,
which means that
there exists a point $\hat x\in X\setminus\{\bx\}$ satisfying
\begin{gather*}
{f}(\hat x)={f}(\hat x)-f(\bx)< \langle\hat x^*,\hat x-\bar{x}\rangle.
\end{gather*}
Thanks to the convexity of $f$, one also has
\begin{gather}\label{ist2}
{f}(x_k)<\langle\hat x^*,x_k-\bar{x}\rangle,\quad k=1,2,\ldots,
\end{gather}
where $x_k:=\bx+\frac{1}{k}(\hat x-\bx)$.
Now select a $z^*\in X^*$ such that $\|z^*\|=1$ and $\ang{z^*,\hat x-\bar{x}}=\|\hat x-\bar{x}\|$, and define 
\begin{align}\label{ral}
g(u):=f(u)+\langle-x^*+\xi z^*,u-\bx\rangle,\quad u\in X.
\end{align}
One has
$\Vert -x^{\ast}+\xi z^*\Vert\leq \Vert x^{\ast }\Vert +\xi <\varepsilon,$
and consequently, $g\in\Ptb_l(f,\bar{x},\varepsilon)$.
It follows from \eqref{ge.05} and \eqref{ral} that
\begin{gather}\label{ral2}
g(u)\ge\xi\langle z^*,u-\bx\rangle\quad \mbox{for all } {u}\in X.
\end{gather}
In particular,
\begin{align}\label{ral3}
g(x_k)\ge\xi\langle z^*,x_k-\bx\rangle =\frac{\xi}{k}\|\hat x-\bar{x}\|>0,\quad k=1,2,\ldots
\end{align}
At the same time, by \eqref{ist2} and \eqref{ral},
\begin{align}\label{qr}
g(x_k)< \langle\hat x^*-x^*+\xi z^*,x_k-\bar{x}\rangle,\quad k=1,2,\ldots
\end{align}
For a fixed $k$, choose a positive $t<1$ such that
$g(x_k)< t\langle\hat x^*-x^*+\xi z^*,x_k-\bar{x}\rangle$ and denote
$$\eta:=t\langle\hat x^*-x^*+\xi z^*,x_k-\bar{x}\rangle>0,\quad\la:=\frac{t}{k}\|\hat x-\bar{x}\|>0.$$
Thus, $g(x_k)<\eta$.
By virtue of the Ekeland variational principle \cite{Eke74} (see also \cite[Theorem~4.3.1]{BorVan10}, \cite[Theorem~4B.5]{DonRoc14}) applied to the function $u\mapsto g_+(u):=\max\{g(u),0\}$, there exists a point $u_k\in X$ such that
\begin{gather}\label{bar1}
\|u_k-x_k\|\le\la,
\\\label{bar3}
g_+(u)+\frac{\eta}{\la}\|u-u_k\|\ge g_+(u_k) \quad \mbox{for all } {u}\in X.
\end{gather}
By \eqref{bar1} and the definitions of $\la$ and $x_k$, we have
\begin{gather}\label{uk}
\|u_k-\bx\|\le\|u_k-x_k\|+\|x_k-\bx\|\le\la+\|x_k-\bx\|< \frac{2}{k}\|\hat x-\bar{x}\|.
\end{gather}
By \eqref{ral2} and \eqref{bar1} and the definitions of $x_k$, $\la$ and $z^*$,
\begin{align}\notag
g(u_k)\ge\xi\langle z^*,u_k-\bx\rangle&=\xi\langle z^*,x_k-\bx\rangle+\xi\langle z^*,u_k-x_k\rangle
\\\label{eqr}
&\ge\frac{\xi}{k}\|\hat x-\bx\|-\xi\la=(1-t)\frac{\xi}{k}\|\hat x-\bx\|>0,
\end{align}
and consequently, $ g_+(u_k)=g(u_k)$.
Thanks to this observation and the fact that $g_+(u)\ge g(u)$ for all $u\in X$, it follows from \eqref{bar3} that $u_k$ is a global minimum of the convex function
$$u\mapsto h(u):=g(u)+\frac{\eta}{\lambda}\left\Vert u-u_k\right\Vert,$$
and so, $0\in \partial h(u_k)=\partial g(u_k)+\frac{\eta}{\lambda}\mathbb{B}^{\ast}$, i.e.,
\begin{align}\notag
d(0,\sd g(u_k))\le\frac{\eta}{\la}&= \frac{\langle\hat x^*-x^*+\xi z^*,x_k-\bar{x}\rangle} {\frac{1}{k}\|\hat x-\bar{x}\|}
\\\notag
&= \frac{k\langle\hat x^*-x^*,x_k-\bar{x}\rangle} {\|\hat x-\bar{x}\|} +\xi
\\\label{eqr2}
&\le\|\hat x^*-x^*\|+\xi<2\xi.
\end{align}
By \eqref{uk}, we have $u_k\to\bx$ as $k\to\infty$.
Thanks to \eqref{eqr} and \eqref{eqr2}, it follows from Theorem~\ref{lo} that $\Er g(\bx)\le2\xi$.
As $\xi>0$ can be chosen arbitrarily small, we conclude that $\Er\{\Ptb_l(f,\bar{x},\eps)\}(\bar{x})=0$,
which proves \eqref{ge.03}.
\qed\end{proof}

\begin{corollary}\label{re}
Let $f\in\Gamma_0(X)$, $f(\bar{x})=0$ and $\varepsilon\ge0$. Then
\begin{align*}
\varepsilon<{|\partial{f}|{}_{\rm bd}(\bar{x})}
&\;\;\Rightarrow\;\;
\Er\{\Ptb(f,\bar{x},\varepsilon)\}(\bar{x})>0
\\
&\;\;\Rightarrow\;\;
\Er\{\Ptb_c(f,\bar{x},\varepsilon)\}(\bar{x})>0
\\
&\;\;\Rightarrow\;\;
\Er\{\Ptb_l(f,\bar{x},\varepsilon)\}(\bar{x})>0
\;\;\Rightarrow\;\;
\varepsilon\le{|\partial{f}|{}_{\rm bd}(\bar{x})}.
\end{align*}
\end{corollary}

\begin{remark}
1. Theorem~\ref{ge} strengthens \cite[Theorem~8 and Corollary~9]{KruNgaThe10} which establish inequality \eqref{ge.02} for any ${|\partial{f}|{}_{\rm bd}(\bar{x})}$ as well as
the first equality in
\eqref{ge.01} in the case ${|\partial{f}|{}_{\rm bd}(\bar{x})}=0$.
The above proof of inequality \eqref{ge.02} is more straightforward than that of the corresponding one in
\cite[Theorem~8~(ii)]{KruNgaThe10}.

2.
Thanks to Theorem~\ref{ge}, the quantity ${|\partial{f}|{}_{\rm bd}(\bar{x})}$ can be interpreted as the \emph{radius of error bounds} for a family of (arbitrary or convex or linear) perturbations of $f$ at $\bar{x}$.

3.
The given above proof of the inequality \eqref{ge.03} is constructive: for an $\eps>{|\partial{f}|{}_{\rm bd}(\bar{x})}$ and an arbitrarily small $\xi>0$, the linear $\varepsilon$-perturbation $g$ of $f$ near $\bar{x}$ satisfying $\Er g(\bx)\le2\xi$
is given by \eqref{ral} where $x^*$ is any element from $\bd\sd f(\bx)$ satisfying $\|x^*\|<\eps$ while element $\hat z^*$ is fully determined by an arbitrary $\hat x^*\notin\sd f(\bx)$ with $\|\hat x^*-x^*\|<\xi$.
\xqed\end{remark}

\section{Stability of global error bounds}\label{S3}

This section deals with the global error bound property for the con\-straint  system (\ref{2}) with a convex function $f:X\rightarrow\R_\infty$.
The next theorem extracted from \cite[Theorem~22]{KruNgaThe10} represents a nonlocal analogue of Theorem~\ref{lo}.

\begin{theorem}\label{li}
Let $f\in\Gamma_0(X)$ and $S_f\ne\emptyset$.
Function $f$ has a global error bound provided that one of the following two conditions is satisfied:
\begin{enumerate}
\item
${|\partial{f}|{}^>}:=\inf_{f(x)>0}d(0,\partial{f}(x))>0$;
\item
${|\partial{f}|{}_{\rm bd}}:=\inf_{f(x)=0}d(0,\bd\partial{f}(x))>0$.
\end{enumerate}
Moreover, condition {\rm (i)} is also necessary for $f$ to have a global error bound and
\begin{equation}\label{li-1}
{|\partial{f}|{}_{\rm bd}}\le{|\partial{f}|{}^>}=\Er{f}.
\end{equation}
\end{theorem}

{We are going to call $|\partial{f}|{}^>$ and $|\partial{f}|{}_{\rm bd}$, respectively, the \emph{outer subdifferential slope} and the \emph{boundary subdifferential slope}
of $f$.}

Obviously,
{for any $x\in X$ with $f(x)=0$,}
${|\partial{f}|{}^>}\le{\overline{|\partial{f}|}{}^>(x)}$ and \begin{equation}\label{auf}
{|\partial{f}|{}_{\rm bd}}=\inf_{x\in S_f^=}{|\partial{f}|{}_{\rm bd}(x)},
\end{equation}
where $S_f^=:=\{x\in X\mid f(x)=0\}$.
The inequality in \eqref{li-1} guarantees, in particular, that $\|x^*\|\ge{|\partial{f}|{}_{\rm bd}}$ for
any $x^*\in\partial{f}(x)$ with $f(x)>0$.
Parts 1 and 2 of Example~\ref{gl} in Section~\ref{On} are also applicable to global error bounds to show that this inequality can be strict.
{We shall prove in this section, that}
condition (ii) in Theorem~\ref{li} corresponds to the existence of a global error bound for a family of functions being small perturbations of $f$.

Alongside ${|\partial{f}|{}_{\rm bd}}$ and ${|\partial{f}|{}^>}$, we are going to consider a ``localized''
{quantity associated with a point $x\in{S}_f^=$ (and depending also on numbers $\eps\ge0$ and $\de\ge0$):
\begin{gather}\label{taux}
\tau(f,x,\eps,\de):=
\begin{cases}
\inf_{f(u)\ge-\eps\|u-x\|-\de} d(0,\partial{f}(u))& \mbox{if } 0\notin\Int\sd f(x),
\\
d(0,\bd\partial{f}(x))& \mbox{if } 0\in\Int\sd f(x).
\end{cases}
\end{gather}
The two cases in definition \eqref{taux} are mutually exclusive in the sense that
{they cannot happen simultaneously for the same function even at different points:}
if $0\in\Int\sd f(x)$ for some $x\in{S}_f^=$, then ${S}_f={S}_f^==\{x\}$, and consequently, there are no points $x\in{S}_f^=$ where $0\notin\Int\sd f(x)$.
In the second case, we are actually dealing with local error bounds at $x$.
This case has been added for completeness to ensure that $\tau(f,x,\eps,\de)$ is defined for all $x\in{S}_f^=$.
For the purposes of the current paper, only the first case is important.}
Unlike ${|\partial{f}|{}_{\rm bd}}$ and ${|\partial{f}|{}^>}$, when
{$0\notin\Int\sd f(x)$ for all $x\in{S}_f^=$ and}
either $\eps>0$ or $\de>0$, this definition takes into account also certain points $u$ with $f(u)<0$.
Obviously,
\begin{gather}\label{30}
\tau(f,x,\eps,\de)\le{|\partial{f}|{}_{\rm bd}}
\end{gather}
for all
{$x\in{S}_f^=$,}
$\eps\ge0$ and $\de\ge0$,
{and the equality holds when $\eps=\de=0$ or $0\in\sd f(x)$}.
Observe that $\tau(f,x,0,\de)$ does not depend on $x$,
{and $\tau(f,x,\eps,\de)$ does not depend on $\eps$ and $\de$ when $0\in\sd f(x)$}.
The function $(\eps,\de)\mapsto\tau(f,x,\eps,\de)$ is nonincreasing in the sense that
$$\tau(f,x,\eps_2,\de_2)\le\tau(f,x,\eps_1,\de_1)$$
if $0\le\eps_1\le\eps_2$ and $0\le\de_1\le\de_2$.

{The next example illustrates the computation and properties of the localized constant \eqref{taux}.
\begin{example}\label{E13}
Consider the piecewise linear function $f:\mathbb{R\rightarrow }\mathbb{R}$ given by
$$
f(x):=\max \{-2x+2,-x+1,2x-5\}=
\begin{cases}
-2x+2 & \mbox{if }x\le1,
\\
-x+1 & \mbox{if }1<x\le2,
\\
2x-5 & \mbox{if }x>2.
\end{cases}
$$
It can be easily verified that $S_{f}^{=}=\{1, 2.5\}$,
${|\partial{f}|{}_{\rm bd}}=1$, ${|\partial{f}|{}^>}=2$, and the only point $u$ where $0\in\sd f(u)$ is $u=2$.
When computing \eqref{taux} with $x=1$ (and nonnegative $\eps$ and $\de$), we see that $2\in\{u\mid f(u)\geq -\varepsilon \left\vert
u-1\right\vert -\delta\}$ if and only if $\varepsilon +\delta\ge1$.
Hence,
\begin{gather}\label{tau1}
\tau (f,1,\varepsilon ,\delta )=\inf_{f(u)\geq -\varepsilon \left\vert
u-1\right\vert -\delta }d(0,\partial f(u))=
\begin{cases}
1 & \mbox{if }\varepsilon +\delta<1,
\\
0 & \mbox{if }\varepsilon +\delta\ge1.
\end{cases}
\end{gather}
For small $\varepsilon$ and $\delta$ (satisfying $\varepsilon +\delta<1$), we have $\tau(f,1,\varepsilon,\delta)={|\partial{f}|{}_{\rm bd}}$.
\xqed\end{example}
}

When examining stability issues of global error bounds, estimates of the difference ${|\partial{f}|{}_{\rm bd}}-\tau(f,x,\eps,\de)$ (which is always nonnegative, cf. \eqref{30}) are needed.
This quantity plays an important role in the next definition.

Let $S_f\neq\emptyset$ and $\varepsilon\ge0$.
We say that
$g:X\to\R_\infty$ is an \emph{$\varepsilon$-per\-turbation} of $f$
if
$S_g\neq\emptyset$,
{$g=f+p$ where $p:X\to\R$ is convex},
and there exist a point $x\in{S}_f^=$ and a number $\xi\ge0$ such that
\begin{gather}\label{Io2}
\xi+{|\partial{f}|{}_{\rm bd}}-\tau(f,x,\xi,
{|p(x)|}
)\le\eps
\end{gather}
and
\begin{gather}\label{Io}
{|p(u)-p(x)|}
\le\xi\|u-x\| \quad\mbox{for all}\quad {u}\in
{X}.
\end{gather}
The collection of all $\varepsilon$-perturbations of $f$ will be denoted $\Ptb(f,\varepsilon)$.

{\begin{remark}
The above definition of an $\varepsilon$-per\-turbation is dependent on the existence of a special point $x\in{S}_f^=$ and contains two main components.
First, the perturbing convex function $p$ is required to be globally calm relative to this point with a small calmness constant $\xi$.
Secondly, the difference ${|\partial{f}|{}_{\rm bd}}-\tau(f,x,\xi,|p(x)|)$ must be small too.
In view of the definition \eqref{taux}, the last condition means that the norms of subgradients of $f$ computed at the points within the set $S_f$ allowed by this definition cannot be much smaller than ${|\partial{f}|{}_{\rm bd}}$.
\xqed\end{remark}}

In view \eqref{30}, condition \eqref{Io2} implies $\xi \leq
\varepsilon $.
Conditions \eqref{Io2} and \eqref{Io} are obviously satisfied with $g\equiv f$, any $x\in{S}_f^=$ and $\xi=0$.
Hence, $f\in\Ptb(f,\varepsilon)$ for any $\eps\ge0$.
If $\tau(f,x,\eps,\de)={|\partial{f}|{}_{\rm bd}}$ for some $x\in{S}_f^=$ and all $\de\ge0$ (for instance, if $0\in\sd f(x)$), then any function $p$ satisfying condition \eqref{Io} with $\xi=\eps$ defines an $\eps$-perturbation $f+p$ of $f$.
\sloppy

Function $p$ in \eqref{Io} enjoys some nice properties.
The next lemma is unlikely to be new, cf. e.g., \cite[Exercise~4.1.28]{BorVan10}.
We provide the proof for completeness.
\begin{lemma}\label{L1}
Suppose $p:X\to\R$ is convex and satisfies condition \eqref{Io} with some $\xi\ge0$.
Then $p$ is Lipschitz continuous with constant $\xi$ and, for any $u\in X$ and any $x^*\in\sd p(u)$, it holds $\|x^*\|\le\xi$.
\end{lemma}
\begin{proof}
Let $u_1,u_2\in X$, $u_1\ne u_2$.
Denote $\la:=\|u_1-u_2\|$ and take an arbitrary $t>0$.
Since,
$$
u_1=\frac{t}{t+\la}u_2+ \frac{\la}{t+\la}\left(u_1+\frac{t}{\la}(u_1-u_2)\right),
$$
using the convexity of $p$ and condition \eqref{Io}, we obtain
\begin{eqnarray*}
p(u_{1}) &\leq &\frac{t}{t+\lambda}p(u_{2})+\frac{\lambda}{t+\lambda}
p\left(u_{1}+\frac{t}{\lambda }(u_{1}-u_{2})\right)  \\
&\leq &\frac{t}{t+\lambda}p(u_{2})+\frac{\lambda}{t+\lambda} \left(\xi\left\Vert u_{1}+\frac{t}{\lambda}(u_{1}-u_{2})-x\right\Vert +p(x)\right)\\
&\leq &\frac{t}{t+\lambda}p(u_{2})+\frac{\lambda}{t+\lambda} \left(\xi\left(\Vert u_{1}\Vert +t+\|x\|\right) +p(x)\right)\\
&=&\frac{t}{t+\lambda}p(u_{2})+\frac{\lambda}{t+\lambda} \left(\xi(\Vert u_{1}\Vert+\|x\|)+p(x)\right) +\frac{\lambda\xi t}{t+\lambda},
\end{eqnarray*}
and after passing to the limit as $t\to\infty$, we conclude that $p(u_1)\le p(u_2)+\la\xi$.
The same inequality obviously holds true with $u_1$ and $u_2$ reversed.
Hence,
$$|p(u_1)-p(u_2)|\le\la\xi=\xi\|u_1-u_2\|,$$
and consequently, $p$ is Lipschitz continuous with constant $\xi$.
If $u\in X$ and $x^*\in\sd p(u)$, then
$$
\ang{x^*,x}\le p(u+x)-p(u)\le\xi\|x\|
$$
for any $x\in X$.
Hence, $\|x^*\|\le\xi$.
\qed\end{proof}

Observe that the above ``global'' definition of an $\varepsilon$-per\-turbation contains a local element: it requires the existence of a point $x\in{S}_f^=$, and the perturbation is defined by \eqref{Io2} and \eqref{Io} relative to this point and depends on the number $\xi$.
If $\xi =0$ in \eqref{Io}, then
{$p(u)=p(x)$}
for all $u\in
{X}.$

Condition \eqref{Io2} imposes a strong restriction on the choice of such a point and a number: for all points $u\in X$ with $f(u)\ge-\xi\|u-x\|-|
{p(x)}
|$, the norms of subgradients of $f$ must not be much smaller than ${|\partial{f}|{}_{\rm bd}}$.
This is also a restriction on the class of functions which admit nontrivial $\varepsilon$-per\-turbations.

\begin{proposition}\label{in}
Let $f\in\Gamma_0(X)$ and $S_f^=\neq\emptyset$.
If there exist sequences $\{x_k\}\subset X$ and $\{x_k^*\}\subset X^*$ such that $x_k^*\in\sd f(x_k)$ $(k=1,2,\ldots)$, $\|x_k\|\to\infty$, $\|x_k^*\|\to0$ as $k\to\infty$ and the sequence $\{f(x_k)\}$ is bounded below, then, for any $\eps\in[0,{|\partial{f}|{}_{\rm bd}}]$ and $g\in\Ptb(f,\varepsilon)$, it holds $g-f={\rm const}$.
\end{proposition}

\begin{proof}
Let $\eps\in[0,{|\partial{f}|{}_{\rm bd}}]$, $g\in\Ptb(f,\varepsilon)$, and sequences $\{x_k\}\subset X$ and $\{x_k^*\}\subset X^*$ satisfy the conditions of the proposition.
Since $g\in\Ptb(f,\varepsilon)$, there are a convex function $p$, an $x\in S_f^=$ and a $\xi\ge0$ satisfying \eqref{Io2} and \eqref{Io}.
If $\xi>0$, then, because $\{f(x_k)\}$ is bounded below, it holds $f(x_k)\ge-\xi\|x_k-x\|-|
{p(x)}
|$ for all sufficiently large $k$, and consequently, $\tau(f,x,\xi,|
{p(x)}
|)=0$.
This contradicts \eqref{Io2} in view of the assumption that $\eps\le{|\partial{f}|{}_{\rm bd}}$.
Hence, $\xi=0$, and it follows from \eqref{Io} that 
{$p(u)=p(x)$}
for all ${u}\in
{X}$.
\qed\end{proof}

\begin{remark}
The conclusion of Proposition~\ref{in} remains valid if condition $\|x_k^*\|\to0$ as $k\to\infty$ is replaced by a weaker one: $\lim_{k\to\infty}\|x_k^*\|\le{|\partial{f}|{}_{\rm bd}}-\eps$.
The assumption that the sequence $\{f(x_k)\}$ is bounded below can be relaxed too: it is sufficient to assume that $\min\{f(x_k),0\}/\|x_k\|\to0$ as $k\to\infty$.
\xqed\end{remark}

The conditions of Proposition~\ref{in} are satisfied, for instance, when the set $S:={\rm argmin}\,f$ is nonempty and unbounded.
Then, obviously, $0\in\sd f(x)$ for all $x\in S$, and any sequence $\{x_k\}\subset S$ with $\|x_k\|\to\infty$ as $k\to\infty$ will do the job.
Below are some specific examples.

\begin{example}\label{gl*}
1. $f(x)\equiv0$, $x\in\R$ (cf. Example~\ref{gl}.1).
Then $S_f=S_f^=={\rm argmin}\,f=\R$,
{$\tau(f,x,\eps,\de)={|\partial{f}|{}_{\rm bd}}=0$ for all $x\in\R$, $\eps\ge0$, $\de\ge0$.}
Thus, $\Ptb(f,0)=\{g\equiv c\mid c\le0\}$.
{If $\varepsilon >0$, then $g\in $
\textrm{Ptb}$(f,\varepsilon )$ if and only if $S_g\neq\emptyset$ and
\[
\left\vert g(u)-g(x)\right\vert \leq \eps\left\vert u-x\right\vert\quad\text{for some }x\in\R \text{ and all }u\in \mathbb{R}\text{.}
\]}

2. $f(x):=\max \{x,-1\}$, $x\in\R$.
Then $S_f=]-\infty,0]$, $S_f^==\{0\}$, ${\rm argmin}\,f=]-\infty,-1]$,
${|\partial{f}|{}_{\rm bd}}=1$, and
{
$$
\tau(f,0,\xi,\de)=
\begin{cases}
0 & \mbox{if }\xi>0,\de\ge0 \mbox{ or } \xi=0,\de\ge1,
\\
1 & \mbox{if }\xi=0,0\le\de<1.
\end{cases}
$$
When $0\le\eps<1$, then condition \eqref{Io2} with $x=0$ is satisfied only when $\xi=0$ and $|g(0)|<1$.
Hence,
$\Ptb(f,\eps)=\{f+c\mid |c|<1\}$.
}



3. $f:\mathbb{R\rightarrow}\mathbb{R}$ is
defined by
\[
f(x):=\left\{
\begin{array}{ll}
-2x & \mbox{if }x\leq 0, \\
-x & \mbox{if }x\in ]0,1], \\
-1-\frac{1}{2}(x-1) & \mbox{if }x\in ]1,2], \\
... & ... \\
-1-\frac{1}{2}-...-\frac{1}{2^{n-1}}-\frac{1}{2^{n}}(x-n) & \mbox{if }%
x\in ]n,n+1], \\
... & ...%
\end{array}%
\right.
\]%
Then $S_f=[0,\infty[$, $S_f^==\{0\}$,
${|\partial{f}|{}_{\rm bd}}=1$, and $\inf f=-2,$ 
Observe that in this example ${\rm argmin}\,f=\emptyset$ and the graph has no horizontal parts.
Nevertheless, it is easy to construct sequences satisfying the assumptions of Proposition~\ref{in}.
{To determine the exact representation of the collection of $\eps$-perturbations, similar to the previous example, $\tau(f,0,\xi,\de)$ needs to be computed:
\[
\tau(f,0,\xi,\de)=\left\{
\begin{array}{ll}
0 & \mbox{if }\xi>0,\de\ge0 \mbox{ or } \xi=0,\de\ge2,\\
1 & \mbox{if }\xi=0,0\le\de<1,\\
\frac{1}{2} & \mbox{if }\xi=0,1\le\de<\frac{3}{2},\\
... & ... \\
\frac{1}{2^{n}} & \mbox{if }\xi=0,2-\frac{1}{2^{n-1}}\le\de<2-\frac{1}{2^{n}},\\
... & ...%
\end{array}%
\right.
\]%
When $0\le\eps<1$, then condition \eqref{Io2} with $x=0$ is satisfied only when $\xi=0$ and $\tau(f,0,0,|g(0)|)\ge1-\eps$.
In particular, if $0\le\eps<1/2$, then it must hold $|g(0)|<1$.
Hence, in this case,
$\Ptb(f,\eps)=\{f+c\mid |c|<1\}$.
}
\xqed\end{example}

In \cite{KruNgaThe10}, when investigating stability of global error bounds, a special \emph{asymptotic qualification condition} $(\mathcal{AQC})$ (generalizing the corresponding finite dimensional condition from \cite{NgaKruThe10}) was imposed on function $f$:
\begin{enumerate}
\item [$(\mathcal{AQC})$]
$\liminf_{k\to\infty}\|x^*_k\|\ge{|\partial{f}|{}_{\rm bd}}$ for any sequences
$x_k\in{X}$ with $f(x_k)<0$ and $x^*_k\in\partial{f}(x_k)$,
$k=1,2,\ldots$, satisfying the following:
\begin{enumerate}
\item
either the sequence $\{x_k\}$ is bounded and
$\lim_{k\to\infty}f(x_k)=0$,
\item
or $\lim_{k\to\infty}\|x_k\|=\infty$ and
$\lim_{k\to\infty}f(x_k)/\|x_k\|=0$.
\end{enumerate}
\end{enumerate}
If $\inf_{x\in X}f(x)<0$, then,
under this condition,
thanks to \cite[Proposition~23]{KruNgaThe10},
for any point $x\in{S}_f^=$ and any $\eps>0$, one can find a positive number $\xi$ such that condition \eqref{Io2} is satisfied as long as $|
{p(x)}
|\le\xi$.
Due to this observation and Lemma~\ref{L1}, under the $(\mathcal{AQC})$ the collection $\Ptb(f,\varepsilon)$ of $\varepsilon$-perturbations defined above is larger than the corresponding one considered in \cite{KruNgaThe10}.

The next example illustrates the computation of $\eps$-perturbations.
\begin{example}
Consider again the piecewise linear function $f:\mathbb{R\rightarrow }\mathbb{R}$ examined in Example~\ref{E13}.
We target small perturbations of this function related to the point $x=1\in{S}_f^=$.
Let $\eps\in[0,1[$.
Since ${|\partial{f}|{}_{\rm bd}}=1$ and $\tau(f,1,\eps,\de)$ equals either 0 or 1 (cf. \eqref{tau1}), condition \eqref{Io2} can only be satisfied when $\tau(f,1,\xi,|
{p}
(1)|)=1$, i.e., when $0\le\xi\le\eps$ and $\xi+|
{p}
(1)|<1$.
Condition \eqref{Io} takes the following form:
\begin{gather*}
|
{p(u)-p(1)}
|\le\xi|u-1| \quad\mbox{for all}\quad {u}\in\R.
\end{gather*}
In particular, when $\eps=0$, all 0-perturbations are of the form $f+c$ with $|c|<1$.
\xqed\end{example}

If condition \eqref{Io2} in the definition of $\varepsilon$-perturbation is dropped we will talk about weak $\varepsilon$-perturbations:
$g:X\to\R_\infty$ is a \emph{weak $\varepsilon$-per\-turbation} of $f$
if
$S_g\neq\emptyset$,
{$g=f+p$ where $p:X\to\R$ is convex},
and there exists a point $x\in{S}_f^=$ such that
\begin{gather*}\label{Io3}
{|p(u)-p(x)|}
\le\eps\|u-x\| \quad\mbox{for all}\quad {u}\in
{X}.
\end{gather*}
In this case, one obviously has $g-p(x)\in\Ptb_c(f,x,\varepsilon)$.
The collection of all weak $\varepsilon$-perturbations of $f$ will be denoted $\Ptb^w(f,\varepsilon)$.
Obviously $\Ptb(f,\varepsilon)\subset\Ptb^w(f,\varepsilon)$.

The following subsets of $\Ptb(f,\varepsilon)$
and $\Ptb^w(f,\varepsilon)$
corresponding to linear $\varepsilon$-per\-turbations of $f$ can be of interest:
\begin{align}\notag
\Ptb_l(f,\varepsilon):=& \{g\mid g(u)-f(u)=\ang{x^*,u-x} (u\in X),
\\\notag
&\qquad x\in{S}_f^=,\; \xi\ge0,\; x^*\in\xi\B^*,\;
\xi+{|\partial{f}|{}_{\rm bd}}-\tau(f,x,\xi,0)\le\eps\},
\\\notag
{\Ptb^w_l(f,\varepsilon):=}&
{\{g\mid g(u)-f(u)=\ang{x^*,u-x} (u\in X),\;
x\in{S}_f^=,\; x^*\in\eps\B^*\}.}
\end{align}
Obviously,
\begin{gather}\label{Rep}
\Ptb_l(f,\varepsilon) \subset
\Ptb(f,\varepsilon),
\quad
\Ptb^w_l(f,\varepsilon) \subset
\Ptb^w(f,\varepsilon),
\\\label{Rep2}
\Ptb^w_l(f,\varepsilon)= \bigcup_{x\in{S}_f^=}\Ptb_l(f,x,\varepsilon).
\end{gather}

Given a function $f\in\Gamma_0(X)$ and a number $\eps\ge0$, denote
\begin{align*}
\Er\{\Ptb(f,\varepsilon)\}:=&
\inf_{g\in\Ptb(f,\varepsilon)}\Er{g},
\\
\Er\{\Ptb_l(f,\varepsilon)\}:=&
\inf_{g\in\Ptb_l(f,\varepsilon)}\Er{g},
\\
{\Er\{\Ptb^w(f,\varepsilon)\}:=}&
{\inf_{g\in\Ptb^w(f,\varepsilon)}\Er{g},}
\\
{\Er\{\Ptb^w_l(f,\varepsilon)\}:=}&
{\inf_{g\in\Ptb^w_l(f,\varepsilon)}\Er{g}.}
\end{align*}
These numbers characterize the error bound property for families of $\eps$-per\-tur\-bations of $f$.
Thanks to \eqref{Rep} and \eqref{Rep2},
it holds for all $\eps\ge0$:
\begin{gather}\label{ert2}
\Er\{\Ptb(f,\varepsilon)\} \le\Er\{\Ptb_l(f,\varepsilon)\}
\le\Er{f},
\\\label{ert22}
{\Er\{\Ptb^w(f,\varepsilon)\} \le\Er\{\Ptb^w_l(f,\varepsilon)\}}= \inf_{x\in{S}_f^=}\Er\{\Ptb_l(f,x,\varepsilon)\}(x).
\end{gather}

\begin{theorem}\label{Ge}
Let $f\in\Gamma_0(X)$, $S_f\ne\emptyset$.
Then
\begin{align}\notag
&\inf\{\eps>0\mid \Er\{\Ptb^w(f,\varepsilon)\}=0\}
\\\notag
\le&\inf\{\eps>0\mid \Er\{\Ptb^w_l(f,\varepsilon)\}=0\}\le{|\partial{f}|{}_{\rm bd}}
\hspace{-.3cm}
&\le
\inf\{\eps>0\mid \Er\{\Ptb(f,\varepsilon)\}=0\}
\\\label{Ge-0}
&&\le\inf\{\eps>0\mid \Er\{\Ptb_l(f,\varepsilon)\}=0\}.
\hspace{-.2cm}
\end{align}
\end{theorem}
\begin{proof}
Thanks to \eqref{ert2}
{and \eqref{ert22}},
we always have
\begin{align}\label{Ge.07}
\inf\{\eps>0\mid \Er\{\Ptb(f,\varepsilon)\}=0\}
\le&\inf\{\eps>0\mid \Er\{\Ptb_l(f,\varepsilon)\}=0\},
\\\label{Ge.07+}
\inf\{\eps>0\mid \Er\{\Ptb^w(f,\varepsilon)\}=0\}
\le&
\inf\{\eps>0\mid \Er\{\Ptb^w_l(f,\varepsilon)\}=0\}.
\end{align}
We first show that
\begin{gather}\label{Ge-2}
{|\partial{f}|{}_{\rm bd}} \le\inf\{\eps>0\mid\Er\{\Ptb(f,\varepsilon)\}=0\}.
\end{gather}
{If ${|\partial{f}|{}_{\rm bd}}=0$, the inequality holds true trivially.}
Let $0<\eps<{|\partial{f}|{}_{\rm bd}}$ and $g\in\Ptb(f,\varepsilon)$,
i.e. conditions \eqref{Io2} and \eqref{Io} are satisfied for some
{convex function $p:X\to\R$, a}
point $x\in{S}_f^=$ and a number $\xi\ge0$.
Two cases are possible.

\underline{If $0\in\Int\partial{f}(\bar{x})$
for some $\bar{x}\in{S}_f^=$}, then
${S}_f={S}_f^==\{\bar{x}\}$,
$x=\bx$,
${|\partial{f}|{}_{\rm bd}}={|\partial{f}|{}_{\rm bd}(\bar{x})}$, and
\begin{gather}\label{Ge-1-}
{|\partial{f}|{}_{\rm bd}}\B^*\subset\sd f(\bar{x}).
\end{gather}
We need to show that
\begin{gather}\label{Ge-1}
\sd f(\bar{x})\subset\sd g(\bar{x})+\xi\B^*.
\end{gather}
Let $x^*\in\partial{f}(\bar{x})$, i.e.,
\begin{gather*}
f(u)-f(\bar{x})-\langle{x}^*,u-\bar{x}\rangle\ge0
\quad\mbox{for all}\quad u\in X.
\end{gather*}
At the same time, by \eqref{Io},
\begin{gather}\label{Io-}
{p(u)-p(\bx)}
\ge-\xi\|u-\bx\| \quad\mbox{for all}\quad {u}\in
{X}.
\end{gather}
Adding the last two inequalities,
we obtain
\begin{gather*}
g(u)-g(\bar{x})+\xi\|u-\bx\|- \langle{x}^*,u-\bar{x}\rangle\ge0
\quad\mbox{for all}\quad u\in X,
\end{gather*}
i.e., $x^*$ is a subgradient at $\bx$ of the sum of two convex functions: $g$ and $u\mapsto\varphi(u):=\xi\|u-\bx\|$.
Function $\varphi$ is Lipschitz continuous and $\sd\varphi(\bx)=\xi{\B}^*$.
Hence, $x^*\in\partial{g}(\bar{x})+\xi{\B}^*$, which proves \eqref{Ge-1}.
Since $\xi\le\eps$, it follows from \eqref{Ge-1-} and \eqref{Ge-1} that
\begin{gather*}\label{Ge-1+}
{|\partial{f}|{}_{\rm bd}}\B^*\subset\sd g(\bar{x})+\eps\B^*.
\end{gather*}
Then, thanks to the \emph{R\aa dstr\"{o}m cancellation principle}
{(cf. Beer \cite[7.4.1]{Bee93}),}
\begin{gather*}\label{Ge-1++}
({|\partial{f}|{}_{\rm bd}}-\eps)\B^*\subset\sd g(\bar{x}),
\end{gather*}
and consequently, $\bx$ is a minimum point of $g$.
Since $S_g\ne\emptyset$, we have $\bx\in S_g$.
We next show that $\|x^*\|\ge{|\partial{f}|{}_{\rm bd}}-\eps$ as long as $x^*\in\sd g(u)$ for some $u\ne\bx$.
Indeed,
for any $u\ne\bx$, $x^*\in\sd g(u)$ and $u^*\in\sd g(\bx)$, one has
\begin{gather*}\label{Ge-3'}
g(u)\geq g(\bx)+\left\langle u^{\ast},u-\bx\right\rangle,
\quad
g(\bx)\geq g(u)+\left\langle x^{\ast},\bx-u\right\rangle.
\end{gather*}
Hence,
\begin{gather*}
\left\langle x^{\ast},u-\bx\right\rangle\ge\langle u^{\ast},u-\bx\rangle,
\end{gather*}
and consequently, taking supremum in the \RHS\ over all $u^*\in({|\partial{f}|{}_{\rm bd}}-\eps)\B^*$,
\begin{gather*}
\left\langle x^{\ast},u-\bx\right\rangle\ge({|\partial{f}|{}_{\rm bd}}-\eps) \|u-\bx\|.
\end{gather*}
Since $u\ne\bx$, it follows immediately from the last inequality that $\|x^*\|\ge{|\partial{f}|{}_{\rm bd}}-\eps$.
Thus, $\|x^*\|\ge{|\partial{f}|{}_{\rm bd}}-\eps$ for all $x^*\in\sd g(u)$ with $u\ne\bx$, particularly, for all $x^*\in\sd g(u)$ with $g(u)>0$.
Hence,
$\Er\{\Ptb(f,\varepsilon)\}\ge{|\partial{f}|{}_{\rm bd}}-\eps>0$ which yields \eqref{Ge-2}.

\underline{$0\notin\Int\partial{f}(x)$ for all $x\in{S}_f^=$}.
If $g(u)>0$, then, in view of \eqref{Io}, $f(u)>-\xi\|u-x\|-
{|p(x)|}
$ and, by \eqref{Io2}, $\|u^*\|\ge{|\partial{f}|{}_{\rm bd}}+\xi-\eps$ for all $u^*\in\partial{f}(u)$.
If $x^*\in\partial{g}(u)$, then,
{in view of Lemma~\ref{L1},
$x^*\in\sd f(u)+\sd p(u)\subset\sd f(u)+\xi\B^*$},
and consequently, $\|x^*\|\ge{|\partial{f}|{}_{\rm bd}}-\eps$.
Hence,
$\Er\{\Ptb(f,\varepsilon)\}\ge{|\partial{f}|{}_{\rm bd}}-\eps>0$, and consequently, \eqref{Ge-2} holds true.
Together with
\eqref{Ge.07}, this proves the second group of inequalities in \eqref{Ge-0}.

Next we show that
\begin{align}\label{Ine}
{\inf\{\eps>0\mid \Er\{\Ptb^w_l(f,\varepsilon)\}=0\}\le{|\partial{f}|{}_{\rm bd}}.}
\end{align}
If ${|\partial{f}|{}_{\rm bd}}=\infty$, the inequality holds trivially.
Let ${|\partial{f}|{}_{\rm bd}}<\eps<\infty$.
By \eqref{auf}, there is an $x\in S_f^=$ such that ${|\partial{f}|{}_{\rm bd}(x)}<\eps$.
By Theorem~\ref{ge}, $\Er\{\Ptb_l(f,x,\varepsilon)\}(x)=0$.
Finally, by the equality in \eqref{ert22},
$\Er\{\Ptb^w_l(f,\varepsilon)\}=0$, and consequently, \eqref{Ine} holds true.
Together with
\eqref{Ge.07+}, this proves the first group of inequalities in \eqref{Ge-0}.
\qed\end{proof}

\begin{corollary}\label{Re}
Let $f\in\Gamma_0(X)$, $S_f\ne\emptyset$, $\varepsilon\ge0$.
Then
\begin{align*}
\Er\{\Ptb^w(f,\varepsilon)\}>0
\;\;&\Rightarrow\;\;
\Er\{\Ptb^w_l(f,\varepsilon)\}>0
\;\;\Rightarrow\;\;
\varepsilon\le{|\partial{f}|{}_{\rm bd}},
\\
\varepsilon<{|\partial{f}|{}_{\rm bd}}
\;\;\Rightarrow\;\;
\Er\{\Ptb(f,\varepsilon)\}>0
\;\;&\Rightarrow\;\;
\Er\{\Ptb_l(f,\varepsilon)\}>0.
\end{align*}
\end{corollary}

\begin{remark}
1. Theorem~\ref{Ge} strengthens \cite[Theorem~25]{KruNgaThe10}.
Nonnegative number ${|\partial{f}|{}_{\rm bd}}$
{provides an estimate for}
the \emph{radius of error bounds} for families of convex or linear perturbations of $f$.

{2. The proof of Theorem~\ref{Ge} does not use the completeness of the underlying space.
The assertion seems to be valid in arbitrary normed linear spaces.}

3. On different stages of the above proof of Theorem~\ref{Ge}, different features of the definition \eqref{Io2}, \eqref{Io} of an $\eps$-perturbation were used.
Accordingly, the statement of the theorem can be strengthened by cutting it into pieces and, for each piece, assuming exactly those features of the definition of an $\eps$-perturbation, that are needed there.

In the proof of the inequality \eqref{Ge-2} when $0\in\Int\partial{f}(\bar{x})$
for some $\bar{x}\in{S}_f^=$, the one-sided estimate \eqref{Io-} was used instead of the `full' inequality \eqref{Io} in the definition of an $\eps$-perturbation while the inequality \eqref{Io2} was not needed at all.
Similarly, in the proof of the same inequality \eqref{Ge-2} in the case when $0\notin\Int\partial{f}(x)$ for all $x\in{S}_f^=$, the opposite one-sided estimate \begin{gather}\label{Io+}
{p(u)-p(\bx)}
\le\xi\|u-\bx\| \quad\mbox{for all}\quad {u}\in
{X}
\end{gather}
was used.

In the proof of the inequality \eqref{Ine} we worked with `weak'
$\eps$-per\-turbations and, again, only the one-sided estimate \eqref{Io+} was used.
\xqed\end{remark}

\section{Concluding remarks and perspectives}
The main results of this article establish the exact formula for the radius of local error bounds (Theorem~\ref{ge}) and estimates of global error bounds (Theorem~\ref{Ge}) for families of convex and linear perturbations of a convex inequality system.

{The next natural step would be to consider possible applications of the obtained estimates in sensitivity analysis of variational problems and convergence analysis of computational algorithms.
Attacking the stability of error bounds for inequality systems defined by non-convex functions is also on the agenda.
The recent article by Zheng \& Wei \cite{ZheWei12} suggests some important classes of such systems to be studied as a starting point.
Investigating nonlinear, particularly H\"older type stability estimates can be of interest.
\begin{acknowledgements}
The authors thank the referees for the careful reading of the manuscript and many constructive comments and suggestions.
{We particularly thank one of the reviewers for attracting our attention to \cite[Theorem~3.2]{Gfr11} and \cite[Theorem~4.4]{LiMor12}.}
\end{acknowledgements}}
\bibliographystyle{spmpsci}
\bibliography{KR-tmp,kruger,buch-kr}
\end{document}